\documentclass{amsart}

\usepackage{amsmath,amsthm,amsfonts,amscd,amssymb} 
\usepackage{verbatim}      	
\usepackage{url}		
\usepackage{epic,eepic,latexsym}
\usepackage{graphicx}
\usepackage{color}
\usepackage{lcd}
\usepackage{mathrsfs}
\usepackage[centering,text={15.5cm,22cm},
        marginparwidth=20mm,dvips]{geometry}
\usepackage{verbatim}
\usepackage[linktocpage]{hyperref}
\usepackage{lscape}
\usepackage{tabularx}
\usepackage{marginnote}
\usepackage[all,color]{xy}
\UseCrayolaColors
\usepackage{enumitem}

\DeclareMathOperator{\Hom}{Hom}

\DeclareMathOperator{\Pic}{Pic}

\DeclareMathOperator{\SPEC}{{\bf Spec}}
\DeclareMathOperator{\Spec}{Spec}

\DeclareMathOperator{\val}{val}

\DeclareMathOperator{\TV}{TV}

\DeclareMathOperator{\SL}{SL}

\DeclareMathOperator{\prin}{prin}

\DeclareMathOperator{\ft}{ft}

\DeclareMathOperator{\coker}{coker}

\DeclareMathOperator{\uf}{uf}
\DeclareMathOperator{\UT}{UT}
\DeclareMathOperator{\Cox}{Cox}
\DeclareMathOperator{\Div}{Div}

\let\?=\overline
\let\:=\colon
\let\bb=\mathbb

\let\rar=\rightarrow

\let\f=\mathfrak
\let\s=\mathcal

\let\wt=\widetilde

\def\risom{\buildrel\sim\over{\smashedlongrightarrow}}
 \def\smashedlongrightarrow{\setbox0=\hbox{$\longrightarrow$}\ht0=1.25pt\box0}

\newcommand {\kk} {\Bbbk}

\theoremstyle{plain}
 \newtheorem{thm}{Theorem}[section]

   \newtheorem{cor}[thm]{Corollary}

\theoremstyle{definition}

 \newtheorem{eg}[thm]{Example}

\theoremstyle{remark} 
 \newtheorem{rmk}[thm]{Remark}
  
  \newtheorem{ass}[thm]{Assumption}

\title[Cluster algebras are Cox rings]{Cluster algebras are Cox rings}
\author{Travis Mandel}
\address{University of Utah\\
Department of Mathematics\\
155 S 1400 E RM 233\\
Salt Lake City, UT, 84112-0090}
\email{mandel{\char'100}math.utah.edu}
\thanks{The author was supported by the National Science Foundation RTG Grant DMS-1246989.}

\begin{document}

\begin{abstract}
It was recently shown by Gross, Hacking, and Keel that, in the absence of frozen indices, a cluster $\s{A}$-variety with generic coefficients is the universal torsor of the corresponding cluster $\s{X}$-variety with corresponding coefficients. We extend this to allow for frozen vectors and corresponding partial compactifications of the $\s{A}$- and $\s{X}$-spaces.  When certain assumptions are satisfied, we conclude that the theta bases of Gross-Hacking-Keel-Kontsevich give bases of global sections for every line bundle on the leaves of the partially compactified $\s{X}$-space.  We note that our arguments work without assuming that the exchange matrix is skew-symmetrizable. 
\end{abstract}

\maketitle

\setcounter{tocdepth}{1}
\tableofcontents

\section{Introduction}

In \cite{FG1}, Fock and Goncharov describe a geometric approach to the study of cluster algebras by defining cluster varieties, denoted $\s{A}$ and $\s{X}$. In \cite{GHK3}, Gross, Hacking, and Keel applied the techniques of birational geometry to the study of cluster varieties to prove a number of powerful results.  One of the more beautiful such results is that, in the absence of frozen vectors, a cluster variety $\s{A}_t$ with generic coefficients $t$ is the universal torsor over the corresponding leaf $\s{X}_{\phi}\subset \s{X}$.  Roughly, this means that $\s{A}_t=\bigoplus_{\s{L}\in \Pic(\s{X}_{\phi})} \s{L}$ (see \S \ref{MainThmSection} for more precise statements).  In particular, this means that the upper cluster algebra with generic coefficients $\Gamma(\s{A}_t,\s{O}_{\s{A}_t})$ is the Cox ring of $\s{X}_{\phi}$.

The main results of this article (Theorems \ref{Pic}, \ref{R}, and \ref{UTor}, and Corollary \ref{CoxCor}) extend the Gross-Hacking-Keel results to allow for frozen vectors and the corresponding partial compactifications of the $\s{A}$- and $\s{X}$-spaces.  Further partial compactifications are handled in Theorem \ref{SigmaThm}.  Such compactifications are important for many geometric and representation-theoretic applications of cluster theory, e.g., those described in \cite[0.4]{GHKK}.

As a consequence of Theorem \ref{R}, we show that under Assumption \ref{FullFG}, every line bundle (up to isomorphism) on our partially compactified $\s{X}$-spaces admits a basis of global sections consisting of certain theta functions from \cite{GHKK}.  See Theorem \ref{ThetaBasis}.  In the language of \cite{GHKK}, Assumption \ref{FullFG} says that the middle and upper cluster algebras of a certain compactification of $\s{A}_{\prin}$ are equal.

As a bonus, we note that all our arguments work without assuming that the exchange matrix from the seed data is skew-symmetrizabe.  The much weaker condition that the non-frozen diagonal entries are $0$ is sufficient.

We note the following previously-known and very important special cases of our results:

\begin{eg}
 The set of spaces $\s{X}_{e}^{\Sigma}$ which can be constructed from a seed with no non-frozen vectors (i.e., $I_{\uf}=\emptyset$) coincides with the set of toric varieties.  In these cases, the realization of $\Gamma(\s{A}^{\Sigma},\s{O}_{\s{A}^{\Sigma}})$ as the Cox ring of $\s{X}_e^{\Sigma}$ agrees with the homogeneous coordinate ring construction of \cite{Cox}, and the theta functions discussed in \S \ref{ThetaLineBundles} are just monomials.
 \end{eg}
 
 \begin{eg}\label{BasicAffine}
 Let $G$ be a semisimple Lie group, $B\subset G$ a Borel subgroup, and $N$ the unipotent radical of $B$.  Then the basic affine space $G/N$ admits a cluster ($\s{A}$-variety) structure via \cite[\S 2.6]{BFZ}, and $G/B$ is the corresponding $\s{X}$-space leaf.  Corollary \ref{CoxCor} in this case is the well-known statement that the coordinate ring of $G/N$ is the homogeneous coordinate ring of $G/B$.  
 
 In the case that $G=\SL_r$, \cite{Magee} shows that Assumption \ref{FullFG} does indeed hold.  Magee concludes that, for a given weight $\lambda$, the integer points of a certain polytope (what we call $\Xi_{\lambda}$) parameterize a theta function basis for the corresponding representation $V_{\lambda}$ of $\SL_r$.  Interpreting $V_{\lambda}$ as a line bundle on $G/B$ gives the conclusion of Theorem \ref{ThetaBasis} applied to this situation.
\end{eg}

\subsection{Acknowledgements}
The author would like to thank Man-Wai Cheung for helpful discussions and for suggesting the project that motivated this work, as well as Sean Keel for helpful feedback.

\section{Construction of cluster varieties}

\subsection{Seeds}

A seed $S$ is data of the form
\begin{align}\label{seed}
S:=(N,I,E:=\{e_i\}_{i\in I},F\subset I,[\cdot,\cdot]),
\end{align}
where $N$ is a lattice of finite rank $n$, $I$ is an index set with $|I|\leq n$, $E$ is a basis for a saturated sublattice $N_I\subseteq N$, $F$ is a subset of $I$, and $[\cdot,\cdot]$ is a $\bb{Z}$-valued bilinear form on $N$ such that $[e_i,e_i]=0$ for all $i\in I_{\uf}:=I\setminus F$.  If $i\in F$, we say $e_i$ is {\bf frozen}.  Let $N_{\uf}\subset N$ denote the span of $\{e_i\}_{i\in I_{\uf}}$. 

Let $M:=N^*=\Hom(N,\bb{Z})$, and let $M_I:=N_I^*$.  Let $\{e_i^*\}_{i\in I}\subset M_I$ denote the dual basis to $E$.  We have two maps $p_1, p_2:N\rar M$ given by $n\mapsto [n,\cdot]$ and $n\mapsto [\cdot,n]$, respectively.  For $i=1,2$, let $\?{p}_i$ denote the composition of $p_i$ with the projection to $M_I$ (i.e., restriction to $N_I$).  Let $K_i:=\ker p_i$, and denote the inclusion $\kappa_i:K_i\hookrightarrow N$.   Let $\langle\cdot,\cdot\rangle:N\times M \rar \bb{Z}$ denote the dual pairing between $N$ and $M$. 

We call $S$ {\bf skew} if $[\cdot,\cdot]$ is skew-symmetrizable, i.e., if there exist positive rational numbers $\{d_j\}_{j\in I_{\uf}}$ and a skew-symmetric form $\{\cdot,\cdot\}$ on $N_{\uf}$ such that $[e_i,e_j]=d_j\{e_i,e_j\}$ for each $i,j\in I_{\uf}$.  Most papers on cluster algebras assume $S$ is skew, but we will not need this.  See \S \ref{FGVersion} for more on the skew cases.

\begin{ass}\label{PrimDistinct}
Unless otherwise stated, we make the following assumptions throughout the paper:
\begin{enumerate}
    \item $p_2(e_i)\neq 0$ for any $e_i\in E$.
    \item $p_2(e_i)$ is primitive in $M$ for each $i\in F$.
\end{enumerate}
\end{ass}

We note that a seed failing Assumption \ref{PrimDistinct}(1) can easily be modified to satisfy that assumption by simply removing from $E$ any $e_i$ in the kernel of $p_2$, and this change would not affect the cluster varieties we construct.

\subsection{Mutations and the construction of cluster varieties}\label{mutaskew}
Given a lattice $L$ with dual lattice $L^*$, let $T_L:=L\otimes \kk^* = \Spec \kk[L^*]$.  For $u\in L$, let $T_{L,u}$ denote the partial compactification $\Spec \kk[v\in L^*|\langle u,v\rangle\geq 0]$ of $T_L$.  We let $D_u$ denote the compactifying divisor $T_{L,u}\setminus T_L$ (or its closure if we further compactify).  A choice of $u\in L$ and $\psi\in L^*$ satisfying $\psi(u)=0$ determines a birational map $\mu_{u,\psi}: T_L \dashrightarrow T_L$ defined by
\begin{align*}
\mu_{u,\psi}^{\sharp}(z^\varphi) :=z^{\varphi}(1+z^{\psi})^{-\varphi(u)} \mbox{~\hspace{.25 in} for $\varphi\in L^*$}.
\end{align*}
$\mu_{u,\psi}$ is called a {\bf mutation}.  See Figure \ref{mutation-fig} for a geometric interpretation of mutation due to \cite{GHK3}.

For a seed $S$, we have a so-called $\s{A}$-torus $A_S:=T_N=\Spec \kk[M]$ and an $\s{X}$-torus  $X_S:=T_M = \Spec \kk[N]$.  For each $j\in I_{\uf}$, define $\mu_{j}^{\s{A}}:A_{S}\dashrightarrow A_{S,j}:=T_N$ and $\mu_{j}^{\s{X}}:X_{S}\dashrightarrow X_{S,j}:=T_M$ by
\begin{align*}
    (\mu_{j}^{\s{A}})^{\sharp}:=\mu_{e_j,p_1(e_j)}^{\sharp}: z^m\mapsto z^m(1+z^{p_1(e_j)})^{-\langle e_j,m\rangle}
\end{align*} 
and
\begin{align*}
    (\mu_{j}^{\s{X}})^{\sharp}:=\mu_{p_2(e_j),e_j}^{\sharp}:z^n\mapsto z^n(1+z^{e_j})^{-[n,e_j]}.
\end{align*}
These are the $\s{A}$- and $\s{X}$-{\bf mutations}, respectively.   For each $j \in F$, we denote partial compactifications $A_{S,j}:= T_{N,e_j}$ and $X_{S,j}:=T_{M,p_2(e_j)}$.   We define what we call {\bf frozen mutations} as follows: for each $j\in F$, $\mu_{j}^{\s{A}}$ and $\mu_{j}^{\s{X}}$ are simply inclusions $A_S\hookrightarrow A_{S,j}$ and $X_S\hookrightarrow X_{S,j}$, respectively.

Given a seed $S$, we define the {\bf cluster $\s{A}$-variety} $\s{A}^{\star}$ (or $\s{A}_S^{\star}$ if $S$ is not clear from context) to be the scheme obtained by gluing $A_S$ to $A_{S,j}$ via $\mu_j^{\s{A}}$ (restricted to the maximal Zariski open subsets on which it is biregular) for each $j\in I$ (including each $j\in F$ using our frozen mutations above).  $A_{S,j}$ and $A_{S,i}$ are similarly glued via $\mu_{i}^{\s{A}}\circ (\mu_j^{\s{A}})^{-1}$ for each $i,j\in I$.

We similarly define the cluster $\s{X}$-variety $\s{X}^{\star}$ (or $\s{X}_S^{\star}$ if $S$ is not clear from context) by gluing $X_S$ to $X_{S,j}$ via $\mu_j^{\s{X}}$ for each $j\in I$.  For each $i,j\in I$, $X_{S,j}$ and $X_{S,i}$ are glued only along the locus on which $\mu_{i}^{\s{X}}\circ (\mu_j^{\s{X}})^{-1}|_{\s{X}_{S,j}\cap \s{X}_S}$ is biregular, even if $\mu_{i}^{\s{X}}\circ (\mu_j^{\s{X}})^{-1}$ is biregular on additional points of $\s{X}_{S,j}$.  As a result, $\s{X}^{\star}$ may be non-separated.

We will write $\s{A}^{\circ}$ and $\s{X}^{\circ}$ to denote the spaces constructed like $\s{A}^{\star}$ and $\s{X}^{\star}$ as above but without applying the mutations for $j\in F$.  I.e., they are the spaces $\s{A}^{\star}$ and $\s{X}^{\star}$, respectively, minus the boundary divisors corresponding to the frozen indices.

The {\bf $\star$-upper cluster algebra} is the ring $\Gamma(\s{A}^{\star},\s{O}_{\s{A}^{\star}})$ of all global regular functions on $\s{A}$.  See \S\ref{FGVersion} for a discussion of how $\s{A}^{\star}$, $\s{X}^{\star}$, and $\Gamma(\s{A}^{\star},\s{O}_{\s{A}^{\star}})$ relate to the usual cluster varieties and upper cluster algebra $\s{A}$, $\s{X}$, and $\Gamma(\s{A},\s{O}_{\s{A}})$, respectively, when $S$ is skew.

\begin{figure}
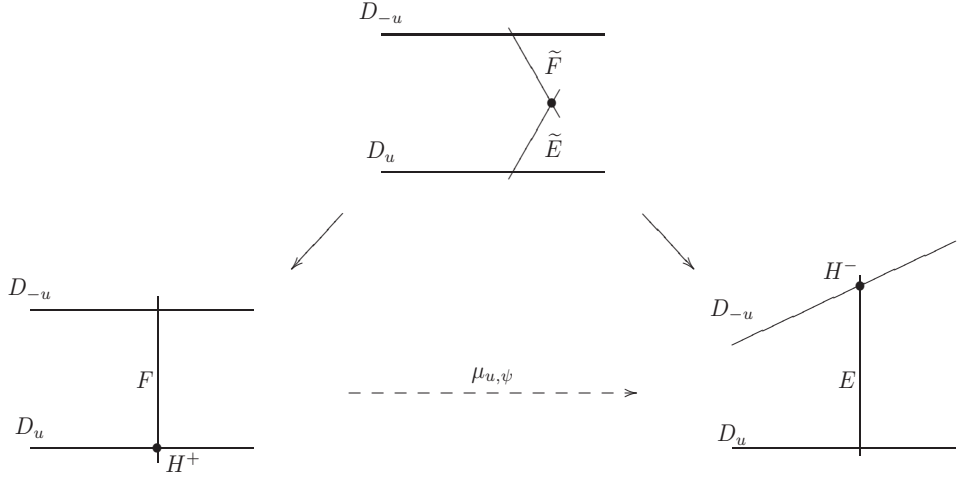

\center{\resizebox{5 in}{1.25 in}{
\xy
(-60,-30)*{}; (-25,-30)*{} **\dir{-}="Dun1";
(-60,-10)*{}; (-25,-10)*{} **\dir{-}="Dup1";
(-40,-8)*{}; (-40,-32)*{} **\dir{-}="F1";
(-60,-27)*{D_{u}};
(-60,-7)*{D_{-u}};
(-42,-20)*{F};
(-5,10)*{}; (30,10)*{} **\dir{-}="Dun2";
(-5,30)*{}; (30,30)*{} **\dir{-}="Dup2";
(15,31)*{}; (23,18)*{} **\dir{-}="F2";
(15,9)*{}; (23,22)*{} **\dir{-}="E2";
(-5,13)*{D_{u}};
(-5,33)*{D_{-u}};
(22,26)*{\wt{F}};
(22,14)*{\wt{E}};
(50,-30)*{}; (85,-30)*{} **\dir{-}="Dun3";
(50,-15)*{}; (85,0)*{} **\dir{-}="Dup3";
(70,-5)*{}; (70,-31)*{} **\dir{-}="E3";
(50,-28)*{D_{u}};
(50,-10)*{D_{-u}};
(68,-20)*{E};
{\ar (-11,4);(-19,-4)};
{\ar (36,4);(44,-4)};
(-40.1,-30)*{\bullet};
(-36,-32)*{H^+};
(21.7,20)*{\bullet};
(70,-6.5)*{\bullet};
(67,-4)*{H^-};
{\ar@{-->}  (-10,-22);(35,-22)};
(12.5,-19.5)*{\mu_{u,\psi}}
\endxy
}}
\caption{Let $u'\in L$ denote the primitive vector in the direction $u$, and let $|u|$ denote the index of $u$, so $u=|u|u'$.  Let $\Sigma$ denote the fan in $L\otimes \bb{Q}$ with rays generated by $u$ and $-u$.  The map $L\rar L/\bb{Z}u'$ induces a $\bb{P}^1$ fibration of the toric variety $\TV(\Sigma)$ over $T_{L/\bb{Z}u'}$. The mutation $\mu_{u,\psi}$ is the birational map $T_L\dashrightarrow T_L$ given geometrically by including $T_L$ into $\TV(\Sigma)$, blowing up the locus $H^+:=D_{u}\cap \?{V}((1+z^{\psi})^{|u|})$ (left arrow),  contracting the proper transform $\wt{F}$ of the fibers $F$ which hit $H^+$ down to a hypertorus $H^-$ in $D_{-u}$ (right arrow), and then taking the complement of the proper transforms of the boundary divisors.  In the figure, $\wt{E}$ denotes the exceptional divisor, with $E$ being its image after the contraction of $\wt{F}$.  The result of gluing the two copies of $T_L$ via $\mu_{u,\psi}$ is given by the top picture minus $D_u$, $D_{-u}$, and $\wt{E}\cap\wt{F}$.}\label{mutation-fig}
\end{figure}

\subsection{An exact sequence of cluster varieties}

Observe that for each seed $S$, there is a {\it not necessarily exact}  sequence 
\begin{align*}
	0 \rar K_2 \stackrel{\kappa_2}{\rar} N \stackrel{p_2}{\rar} M \stackrel{\lambda}{\rar} K_1^* \rar 0.
\end{align*}
Here, $\lambda:M\rar K_1^*$ is the dual to the inclusion $\kappa_1$.  Tensoring with $\kk^*$ yields an exact sequence
\begin{align*}
	1 \rar T_{K_2} \stackrel{\kappa_2}{\rar} T_N \stackrel{p_2}{\rar} T_M \stackrel{\lambda}{\rar} T_{K_1^*} \rar 1.
\end{align*}
In terms of functions, these maps are given by $\kappa_2^{\sharp}(z^m)=z^{m|_{K_2}}$, $p_2^{\sharp}(z^n)=z^{p_1(n)}$, and $\lambda^{\sharp}(z^{k})=z^{\kappa_1(k)}$.  With these descriptions, it is straightforward to check that the sequence commutes with mutations, including extending over the boundary divisors from our frozen mutations (cf. \cite[Lem. 2.10]{FG1} for the skew cases without the boundary divisors).  Thus, one obtains an exact sequence 
\begin{align}\label{CES}
	1\rar T_{K_2} \stackrel{\kappa_2}{\rar} \s{A}^{\star} \stackrel{p_2}{\rar} \s{X}^{\star} \stackrel{\lambda}{\rar} T_{K_1^*} \rar 1.
\end{align}
  We note that we will repeatedly abuse notation as above, using the same notation to denote maps of lattices, the induced maps of tori, and also the induced maps of cluster varieties. 

For $\phi\in T_{K_1^*}$, let $\s{X}^{\star}_{\phi}$ denote the fiber of $\lambda$ over $\phi$.  In particular, $p_2(\s{A}^{\star})=\s{X}^{\star}_e$ where $e$ denotes the identity in $T_{K_1^*}$.

\subsection{Principal coefficients}\label{prin}

We will also need {\bf  cluster varieties with principal coefficients}.  For $S$ as in \eqref{seed}, define 
\[
S_{\prin}:=(N_{\prin},I,E_{\prin},F,[\cdot,\cdot]_{\prin})
\]
by
\begin{itemize}[noitemsep]
\item $N_{\prin}=N_I\oplus M$, 
\item $I$ and $F$ are the same as for $S$,
\item $E_{\prin}:=(E,0)$, by which we mean $\{(e_i,0)|i\in I\}$.
\item $[(n_1,m_1), (n_2,m_2)]_{\prin} : = [n_1,n_2] + \langle n_1,m_2\rangle - \langle n_2,m_1\rangle$.  Here, $\langle\cdot,\cdot\rangle$ is the restriction of the dual pairing on $N\oplus M$ to $N_I\oplus M$.
\end{itemize}
Then $\s{A}^{\star}_{\prin}$ and $\s{X}^{\star}_{\prin}$ are simply the spaces $\s{A}_{S_{\prin}}^{\star}$ and $\s{X}_{S_{\prin}}^{\star}$ constructed with respect to the seed $S_{\prin}$.  Let $p_{1,\prin}$ and $p_{2,\prin}$ denote the corresponding $p_1$ and $p_2$.  For future use, we note that $p_{1,\prin}((n,0))=(\?{p}_1(n),n)$, so 
\begin{align}\label{muAprin}
    (\mu_j^{\s{A}_{\prin}})^{\sharp}:z^{(m,n)} \mapsto z^{(m,n)}(1+z^{(\?{p}_1(e_j),e_j)})^{-\langle e_j,m\rangle}.
\end{align}
In particular, $\mu_j^{\s{A}_{\prin}}$ commutes with the grading 
\begin{align}\label{degprin}
    \deg(z^{(m,n)}) := m-\?{p}_1(n)\in M_I.
\end{align} Hence, $\deg$ induces an $M_I$-grading on $\s{O}_{\s{A}^{\star}_{\prin}}$.  We note that $\deg$ is denoted $w$ (for weight) in \cite{GHKK} (cf. their Prop. 7.7), as $z^{q}$ is a weight $\deg(z^q)$ eigenfunction for the action of $T_{N_I}$ on $\s{A}_{\prin}^{\star}$.

Define $\pi:N_{\prin}\rar M$, $(n,m)\mapsto m$.  Also, define $\wt{p}_2:N_{\prin}\rar M$, $(n,m)\mapsto p_2(n)+m$. 
 Let $\lambda$ and $i$ both denote the dual to $\kappa_1$.  We have a commutative diagram:
\begin{align*}
    \xymatrix{
N_{\prin} \ar[d]_{\pi}  \ar[r]^{\wt{p}_2} &M \ar[d]^{\lambda}\\
M  \ar[r]_{i} & K_1^* }
\end{align*}
Tensoring with $\kk^*$ yields a commutative diagram of maps of tori.  In terms of functions, $\pi^{\sharp}(z^n)=z^{(0,n)}$, $\wt{p}_2^{\sharp}(z^n)=z^{(\?{p}_1(n),n)}$, and $i^{\sharp}$ and $\lambda^{\sharp}$ are both given by $\lambda^{\sharp}(z^{k})=z^{\kappa_1(k)}$.  One now checks that the top row commutes with mutations (including extending over the boundary divisors from our frozen mutations), and so we obtain a commutative diagram:
\begin{align}\label{AprinCommutes}
    \xymatrix{
\s{A}^{\star}_{\prin} \ar[d]_{\pi}  \ar[r]^{\wt{p}_2} &\s{X}^{\star} \ar[d]^{\lambda}\\
T_M  \ar[r]_{i} & T_{K_1^*} }
\end{align}
For each $\phi \in T_{K_1^*}$ and $t\in i^{-1}(\phi)$, let $\s{A}_{\phi}^{\star}$ denote the fiber of $i\circ \pi$ over $\phi$, and let $\s{A}_t^{\star}$ denote the fiber of $\pi$ over $t$.  Then $\wt{p}_2$ restricts to a map $p_{2,\phi}:\s{A}^{\star}_{\phi}\rar \s{X}^{\star}_{i(t)}$ and further restricts to $p_{2,t}:\s{A}^{\star}_t\rar \s{X}^{\star}_{i(t)}$.  We note for future reference that $\s{A}_t^{\star}$ is the subscheme of $\s{A}^{\star}_{\prin}$ corresponding to the ideal sheaf
\begin{align}\label{It}
\s{I}_t:=\langle z^{(m,n+n_0)}-z^{(m,n)}z^{n_0}(t)|m\in M_I, n,n_0\in N\rangle.
\end{align}
Here, $z^{n_0}(t)$ means the evaluation of the monomial $z^{n_0}$ at the point $t\in T_M$.  Similarly, $\s{X}_{\phi}^{\star}$ is the subscheme of $\s{X}^{\star}$ corresponding to
\begin{align}\label{Iphi}
\s{I}_{\phi}:=\langle z^{n+k}-z^nz^{k}(\phi)| n\in N,k\in K_1\subset N\rangle.
\end{align}

\subsection{Further compactifications}\label{compact}

Let us assume that $p_2(e_i)\neq p_2(e_j)$ for distinct $i,j\in F$.  Let $\Sigma$ be a fan in $M_{\bb{R}}$ whose set of rays $\Sigma^{[1]}$ is $\{\bb{R}_{\geq 0} p_2(e_i): i\in F\}$.  For a cone $\sigma\in \Sigma$ whose bounding rays are generated by $p_2(e_{i_1}),\ldots,p_2(e_{i_k})$, $i_j\in F$, let $\wt{\sigma}$ denote the cone in $N_{\bb{R}}$ whose bounding rays are generated by $e_{i_1},\ldots,e_{i_k}$.  Let $\wt{\Sigma}$ be the fan $\{\wt{\sigma}|\sigma \in \Sigma\}$ in $N_{\bb{R}}$.

Now, in the construction of $\s{X}^{\star}$, in addition to gluing $X_S=T_M$ to $X_{S,i}$ for each $i\in I$, we can also attach $X_{S,\sigma}:=\Spec \kk[\sigma^{\vee}]$ for each $\sigma\in \Sigma$.  For example, in the case where $\sigma$ is the ray generated by $p_2(e_i)$, $i\in F$, $X_{S,\sigma}$ is the same as $X_{S,i}$.  We denote the resulting scheme by $\s{X}^{\Sigma}$.

We analogously define $\s{A}^{\Sigma}$, but using $\wt{\Sigma}$ in place of $\Sigma$.  Similarly for $\s{A}^{\Sigma}_{\prin}$ using the natural inclusion $N\hookrightarrow N\oplus M$ to view $\wt{\Sigma}$ as a fan in $N_{\prin}$.

We note that the fibration $\lambda$ of $\s{X}^{\star}$ over $T_{K_1^*}$ extends to $\s{X}^{\Sigma}$.  Similarly, the fibrations $\pi$ and $i\circ \pi$ of $\s{A}^{\star}_{\prin}$ over $T_M$ and $T_{K_1^*}$, respectively, extend over $\s{A}_{\prin}^{\Sigma}$.  We thus obtain fibers $\s{X}^{\Sigma}_{\phi}$, $\s{A}^{\Sigma}_{t}$, and $\s{A}^{\Sigma}_{\phi}$.

\subsection{Relation to the usual notion of cluster varieties}\label{FGVersion}
If $S$ is skew, then for each $j\in I_{\uf}$, we can define a new skew seed $\mu_j(S)$ given by the same data as $S$ except for $E$, which is replaced with $E':=\{\mu_j(e_i)\}_{i\in I}$, where \begin{align*}
\mu_j(e_i):=\begin{cases}
e_i+\max([e_i,e_j],0) e_j &\mbox{~if $i\neq j$} \\
-e_i &\mbox{~if $i= j$.}
\end{cases}
\end{align*}
One can mutate again with respect to any $j\in I_{\uf}$, and in this way one obtains an infinite oriented rooted tree $\f{T}$ of skew seeds, with $|I_{\uf}|$ outgoing edges (labelled by the elements of $I_{\uf}$) at each vertex corresponding to the different possible mutations of the corresponding seed.\footnote{Skew-ness of $S$ ensures that repeated mutation preserves the assumption that $[e_i,e_i]=0$ for all $i\in I_{\uf}$.}  The cluster $\s{A}$-variety $\s{A}$ of \cite{FG1} (see also \cite[\S 2]{GHK3}) is then constructed by gluing the $\s{A}$-tori corresponding to the different vertices together using the compositions of the $\s{A}$-mutations associated to the paths between the vertices.  Similarly for the cluster $\s{X}$-variety $\s{X}$, using the $\s{X}$-tori and $\s{X}$-mutations.  The superscript $\ft$ is added if one restricts to a finite connected regular subtree which includes $S$ and the adjacent seeds.

Our boundary divisors associated to the frozen vectors may be added by applying all the gluings associated to the frozen mutations for each $j\in F$ at any vertex of $\f{T}$.  More generally, we could glue boundary strata corresponding to cones in a fan as in \S \ref{compact}.  The case where $\wt{\Sigma}$ includes the cone generated by $\{e_i|i\in F\}$ and all its subcones agrees with the treatment of frozen vectors in \cite[Construction 2.10]{GHK3}.

Our $\s{A}^{\star}$ and $\s{X}^{\star}$ in the skew cases are clearly subschemes of this $\s{A}$ and $\s{X}$.  Without the frozen mutations, they are the minimal cases of $\s{A}^{\ft}$ and $\s{X}^{\ft}$ from \cite[Rmk. 2.6]{GHK3}.  In fact, in the absence of frozen vectors, our $\s{X}^{\star}$ is what \cite[\S 4]{GHK3} denotes as $X$, and similarly, our $\s{A}_{\prin}^{\star}$ and $\s{A}_t^{\star}$ agree with their $A_{\prin}$ and $A_t$, respectively.

\cite[Theorem 3.9]{GHK3} says that $\s{X}^{\star}$ actually covers all but a codimension $2$ subset of any $\s{X}^{\ft}$.  If $t\in T_{M}$ is very general,\footnote{One says that $t$ is very general if it is in the complement of some countable union of codimension $1$ subsets.} the same theorem says that the analogous statement holds for $\s{X}^{\star}_t$ with $\s{X}_t^{\ft}$ and also for $\s{A}^{\star}_t$ with $\s{A}_t^{\ft}$.  Similarly for $\s{A}^{\star}$ and $\s{A}^{\ft}$ if $S$ is totally coprime.\footnote{A skew seed $S$ is called coprime if $p_2(e_i)\neq rp_2(e_j)$ for any distinct $i,j\in I_{\uf}$ and $r\in \bb{Q}_{>0}$.  $S$ is called totally coprime if every seed mutation equivalent to $S$ is coprime.  In particular, $S$ is totally coprime whenever $[\cdot,\cdot]$ is non-degenerate.}  This agreement up to codimension $2$ ensures that spaces of global sections of line bundles on the pairs of spaces agree.  In particular, our $\star$-upper cluster algebra agrees with the usual upper cluster algebra $\Gamma(\s{A},\s{O}_{\s{A}})$ for totally coprime skew seeds, but not for all skew seeds.

\section{Picard groups, universal torsors, and Cox rings}\label{MainThmSection}
For convenience, for each $i\in I$, we will start writing $X_{S,i}$ as simply $U_i$, and we will write $X_S$ as $U_0$.  For each $i\in I_{\uf}$, let $E_i$ denote the exceptional divisor associated to $\mu_i^{\s{X}}$.  That is, $E_i=U_{i}\setminus (U_0\cap U_{i})\subset \s{X}^{\star}$.  Similarly, for each $i\in F$, let $D_i$ denote the boundary divisor $D_{p_2(e_i)}=U_{i}\setminus U_0$.  We can view the $E_i$'s and $D_i$'s as Weil divisors in $\s{X}^{\star}$ or $\s{X}^{\Sigma}$.  Let $E_{i,\phi}$ and $D_{i,\phi}$ denote the corresponding intersections with $\lambda^{-1}(\phi)$.

\begin{rmk}\label{WeilCartPic}
When defining Weil divisors, \cite[\S II.6]{Hart} assumes the scheme under consideration is a Noetherian integral separated scheme which is regular in codimension one.  In general though, $\s{X}^{\star}$, $\s{X}^{\Sigma}$, and some fibers $\s{X}^{\star}_{\phi}$ and $\s{X}^{\Sigma}_{\phi}$ may fail to be separated.  However, these spaces are all Noetherian,\footnote{It is not clear whether or not the cluster varieties $\s{A}$ and $\s{X}$ for skew $S$, as defined in \S \ref{FGVersion}, are Noetherian.  This is why \cite{GHK3} frequently restricts to $\s{A}^{\ft}$ and $\s{X}^{\ft}$, cf. their Remark 2.6.} integral, and regular in codimension one since the spaces $U_i$, $i\in I\cup \{0\}$ are.  This is sufficient for the definition of Weil divisors and principal Weil divisors to make sense.  Furthermore, since each $U_i$ is locally factorial, so are $\s{X}^{\star}$ and each $\s{X}_{\phi}^{\star}$, as well as $\s{X}^{\Sigma}$ and each $\s{X}_{\phi}^{\Sigma}$ if $\Sigma$ is a non-singular fan (meaning that the rays of any cone of $\Sigma$ are generated by part of a basis for the lattice).  This is sufficient for the correspondence between Weil divisors and Cartier divisors as in \cite[Prop. 6.11]{Hart}, as well as the correspondence between Cartier divisor classes and isomorphism classes of line bundles as in \cite[Prop. 6.15]{Hart}.
\end{rmk}

\begin{rmk}\label{Ii}
It will be useful to have a description of the above divisors in terms of equations.  Let $\s{I}_i$ denote the ideal sheaf on $\s{X}^{\star}$ for $E_i$ if $i\in I_{\uf}$ or $D_i$ if $i\in F$.  Then for any $i\in I$ and $j\neq i$ in $I\cup \{0\}$, we have $\s{I}_i(U_{j})=\langle 1 \rangle$.  If $i\in I_{\uf}$, then $\s{I}_i(U_i)=\langle 1+z^{e_i} \rangle \subset \s{O}_{\s{X}^{\star}}(U_i)=\kk[N]$.  If $i\in F$, then $\s{I}_i(U_i)=\langle z^n|[n,e_i]>0\rangle \subset \s{O}_{\s{X}^{\star}}(U_i) = \kk[n\in N|[n,e_i]\geq 0]$.  The same equations reduced modulo $\s{I}_{\phi}$ give $E_{i,\phi}$ and $D_{i,\phi}$.
\end{rmk}

We define maps $W:M_I\hookrightarrow \Div(\s{X}^{\star})$ ($\Div$ denoting the group of Weil divisors) and $W_{\phi}:M_I\hookrightarrow \Div(\s{X}_{\phi}^{\star})$ as follows:
\begin{align}\label{Wm}
W(\sum_{i\in I} a_ie_i^{*}) &:=\sum_{i\in I_{\uf}} a_i E_i + \sum_{i\in F} a_i D_i \\
W_{\phi}(\sum_{i\in I} a_ie_i^{*}) &:=\sum_{i\in I_{\uf}} a_i E_{i,\phi} + \sum_{i\in F} a_i D_{i,\phi}. \nonumber
\end{align}

\begin{thm}\label{Pic}
The map $m\mapsto \s{O}_{\s{X}^{\star}}(W(m))$ induces an identification of $\coker(\?{p}_1)$ with $\Pic(\s{X}^{\star})$.  Similarly, for any $\phi \in T_{K_1^*}$, $m\mapsto \s{O}_{\s{X}_{\phi}^{\star}}(W_{\phi}(m))$ induces an identification of $\coker(\?{p}_1)$ with $\Pic(\s{X}_{\phi}^{\star})$.
\end{thm}
\begin{proof}
We do the proof for $\s{X}^{\star}$, but the argument for $\s{X}^{\star}_{\phi}$ is the same.  $U_0=\s{X}^{\star}\setminus \left(\bigcup_{i\in I_{\uf}} E_i \cup \bigcup_{i\in F} D_i\right)$ is affine with a UFD for its coordinate ring.  Hence, as in the proof of the proposition on pg 63 of \cite{Fult}, $W(M_I)$ surjects onto the divisor class group of $\s{X}^{\star}$.   Given $n\in N$, let $z^n$ be corresponding monomial on $U_0$, extended to be rational function on $\s{X}^{\star}$.  We have $\val_{E_i}(z^n) = [n,e_i]$ for each $i\in I_{\uf}$ since $[(\mu_i^{\s{X}})^{-1}]^{\sharp}(z^n)=z^n(1+z^{e_i})^{[n,e_i]}$, and by Remark \ref{Ii}, $1+z^{e_i}$ cuts out $E_i$ in $U_i$.  We also see from Remark \ref{Ii} that $\val_{D_i}(z^n)=[n,e_i]$ for each $i\in F$.  Hence, for each $n\in N$, the principal Weil divisor $(z^n)$ is $W(\sum_{i\in I} [n,e_i] e_i^*)$, i.e., $W(\?{p}_1(n))$.   All principal Weil divisors supported on $W(M_I)$ are of this form (non-monomial rational functions will have zeroes and/or poles in $U_0$), so the group of divisor classes is given by $W(M_I)/W(\?{p}_1(N)) \cong \coker(\?{p}_1)$.  The claim follows using the correspondence between the divisor class group and $\Pic$ as discussed in Remark \ref{WeilCartPic}.
\end{proof}

\begin{rmk}
The case with no frozen variables and $N_I=N$ is \cite[Thm. 4.1]{GHK3}.  Their argument is directly in terms of \v{C}ech cocycles, using the identification $\Pic(\s{X}^{\star})=H^1(\s{X}^{\star},\s{O}_{\s{X}^{\star}}^{\times})$ (and similarly with $\phi$).  This alternative argument generalizes to our setup using $\{U_i\}_{i\in I\cup \{0\}}$ as the \v{C}ech cover.
\end{rmk}

Now, following \cite[Construction 4.3]{GHK3} (slightly modified), we review the construction of a {\bf universal torsor} as in \cite{BH} (due to \cite{HK} in cases where $\Pic$ is torsion free).  Let $X$ be a scheme (either $\s{X}^{\star}$ or $\s{X}^{\star}_{\phi}$ for us).  In the case that $\Pic(X)$ is a free Abelian group, we can simply take
\begin{align*}
\UT_{X}:=\SPEC \bigoplus_{\s{L}\in \Pic X} \s{L}.
\end{align*}
In general though there is not a natural product on this direct sum, so we use the following construction.  Let $\Lambda$ denote a lattice of Cartier divisors on $X$ such that the map $\Lambda \rar \Pic(X)$, $m \mapsto [\s{O}_X(m)]$, is surjective.  Let $\Lambda_0$ denote the kernel of this map.  For our purposes, we will take $\Lambda:=M_I$ and $\Lambda_0:=\?{p}_1(N)$, identified with a lattice of divisors via $W$ or $W_{\phi}$ as in \eqref{Wm}.  In particular, for $m\in M_I$, we may write $m$ to mean $W(m)$ or $W_{\phi}(m)$.

Now, for each $m \in \Lambda$, viewing sections of $\s{O}_X(m)$ as elements of the function field $K(X)$ makes
\begin{align*}
    \s{R}_{X,\Lambda}:=\bigoplus_{m\in \Lambda} \s{O}_{X}(m)z^{m}
\end{align*}
into a locally free sheaf of $\Lambda$-graded $\s{O}_{X}$-algebras (here, $z^mz^{m'}:=z^{m+m'}$ for any $m,m'\in \Lambda$). 

Recall from \S \ref{prin} the map $\wt{p}_2:\s{A}_{\prin}^{\star}\rar \s{X}^{\star}$ which, when restricted to $A_{S_{\prin}}\rar X_S$, is given by $\wt{p}_2^{\sharp}:\kk[N]\rar \kk[M_I\oplus N]$, $z^n\mapsto z^{(\?{p}_1(n),n)}$.  We extend this map by defining
\begin{align}\label{p2sharp}
    \wt{p}_2|_{X_S}^{\sharp}:\kk[M_I][N]\rar \kk[M_I\oplus N], \hspace{.25 in} z^mz^n\mapsto z^{(m+\?{p}_1(n),n)},
\end{align}  
and defining $\wt{p}_2|_{X_{S,i}}^{\sharp}$ to be given by $[(\mu_i^{\s{A}_{\prin}^{\star}})^{-1}]^{\sharp}\circ \wt{p}_2|_{X_S}^{\sharp} \circ (\mu_i^{\s{X}})^{\sharp}$.  As before, we write $p_{2,\phi}:=\wt{p}_2|_{\s{A}_{\phi}^{\star}}$.
\begin{thm}\label{R}
$\s{A}_{\prin}^{\star}=\SPEC \s{R}_{\s{X}^{\star},M_I}$.  More precisely, $\wt{p}_2^{\sharp}$ induces an isomorphism 
\begin{align}\label{Req}\wt{p}_2^{\sharp}:\s{R}_{\s{X}^{\star},M_I} \risom (\wt{p}_2)_*\s{O}_{\s{A}_{\prin}^{\star}}
\end{align} 
of $M_I$-graded $\s{O}_{\s{X}^{\star}}$-algebras.  Similarly, $\s{A}_{t}^{\star}=\SPEC \s{R}_{\s{X}_{\phi}^{\star},M_I}$ in the sense that $p_{2,\phi}^{\sharp}$ induces an isomorphism
\begin{align}\label{Rphi} p_{2,\phi}^{\sharp}:\s{R}_{\s{X}_{\phi}^{\star},M_I} \risom (p_{2,\phi})_*\s{O}_{\s{A}_{\phi}^{\star}}
\end{align} 
\end{thm}
\begin{proof}
We first prove the $\prin$ version of the claim.  Recall that $\s{O}_{\s{A}_{\prin}^{\star}}$ is $M_I$-graded via $\deg(z^{(m,n)})=m-\?{p}_1(n)$, so the degree $m$ part $\s{O}_{\s{A}_{\prin}^{\star}}(m)$ is spanned by the terms of the form $z^{(m+\?{p}_1(n),n)}$ for $n\in N$.  We want to show that for each $m\in M_I$, $\wt{p}_2^{\sharp}|_{\s{O}_{\s{X}^{\star}}(m)}$ gives an isomorphism of $\s{O}_{\s{X}^{\star}}$-modules $\s{O}_{\s{X}^{\star}}(m)\cdot z^m \risom (\wt{p}_2)_*[\s{O}_{\s{A}_{\prin}^{\star}}(m)]$.

We use the affine cover $\{U_i\}_{i\in I\cup \{0\}}$ of $\s{X}^{\star}$ and the ideal sheaves from Remark \ref{Ii}.  For each $i\in I\cup \{0\}$, we let $\wt{U}_i=\wt{p}_2^{-1}(U_i)$ denote the corresponding preimage in $\s{A}_{\prin}^{\star}$.  Note that $\wt{U}_i=A_{S_{\prin},i}$ for each $i\in I$ and $\wt{U}_0=A_{S_{\prin}}$.  By definition, $(\wt{p}_2)_*\s{O}_{\s{A}_{\prin}^{\star}}(U_i)=\s{O}_{\s{A}_{\prin}^{\star}}(\wt{U}_i)$ for each $i\in I\cup \{0\}$.

One easily sees that $\wt{p}_2^{\sharp}$ is injective, so it suffices to check that $$\wt{p}_2^{\sharp}\left(\s{O}_{\s{X}^{\star}}(m)(U_i)\cdot z^m\right) =\s{O}_{\s{A}_{\prin}^{\star}}(\wt{U}_i)$$ for each $i\in I\cup \{0\}$.  For $i=0$, we have 
\begin{align*}
    \wt{p}_2^{\sharp}\left(\s{O}_{\s{X}^{\star}}(m)(U_0)\cdot z^m\right) &= \wt{p}_2^{\sharp} \left(\bigoplus_{n\in N} \kk\cdot (z^n) z^{m} \right) \\
                                                  &= \bigoplus_{n\in N} \kk\cdot z^{(m+\?{p}_1(n),n)} \\
                                                  &= \s{O}_{\s{A}_{\prin}^{\star}}(m)(\wt{U}_0).
\end{align*}

Now suppose $i\in I_{\uf}$.  $\wt{p}_2^{\sharp}$ as defined in \eqref{p2sharp} does not commute with mutations, so we will have to apply $[(\mu_i^{\s{X}})^{-1}]^{\sharp}\circ \wt{p}_2^{\sharp} \circ (\mu_i^{\s{X}})^{\sharp}$.  In the copy of the coordinate system $\kk[N]$ used for defining $U_i=X_{S,i}=T_M$, we have $$\s{O}_{\s{X}^{\star}}(m)(U_i)=\bigoplus_{n\in N} \bigoplus_{k\geq -\langle e_i,m\rangle} z^m z^n(1+z^{e_i})^k.$$  We apply $(\mu_i^{\s{X}})^{\sharp}$ to express this in terms of the coordinates for $U_0$, resulting in:
\begin{align*}
\s{O}_{\s{X}^{\star}}(m)(U_i)=\bigoplus_{n\in N} \bigoplus_{k\geq -\langle e_i,m\rangle} \kk\cdot z^m z^n(1+z^{e_i})^{-[n,e_i]+k}.
\end{align*}
Now applying $\wt{p}_2^{\sharp}$ yields
\begin{align*}
\wt{p}_2^{\sharp} \left(\s{O}_{\s{X}^{\star}}(m)(U_i)\right)&=\bigoplus_{n\in N} \bigoplus_{k\geq -\langle e_i,m\rangle} z^{(m+\?{p}_1(n),n)}(1+z^{(\?{p}_1(e_i),e_i)})^{-[n,e_i]+k} \\
&=\bigoplus_{n\in N} \bigoplus_{\ell \geq 0} z^{(m+\?{p}_1(n),n)}(1+z^{(\?{p}_1(e_i),e_i)})^{-\langle e_i,m+\?{p}_1(n)\rangle+\ell}.
\end{align*}
Using \eqref{muAprin}, we see that applying $((\mu_i^{\s{A}_{\prin}})^{\sharp})^{-1}$ transforms this to
\begin{align*}
\bigoplus_{n\in N} \bigoplus_{\ell \geq 0} z^{(m+\?{p}_1(n),n)}(1+z^{(\?{p}_1(e_i),e_i)})^{\ell},
\end{align*}
and this is indeed equal to $\s{O}_{\s{A}_{\prin}^{\star}}(m)(\wt{U}_i)$.

Finally, suppose $i\in F$.  Then 
\begin{align*}
\wt{p}_2^{\sharp}\left(\s{O}_{\s{X}^{\star}}(m)(U_i)\right)&=\wt{p}_2^{\sharp}\left(\bigoplus_{\substack{n\in N \\ [n,e_i]\geq -\langle e_i,m\rangle}} \kk\cdot z^m z^n \right)\\
                             &=\bigoplus_{\substack{n\in N \\ \langle e_i,m+\?{p}_1(n)\rangle \geq 0}} \kk\cdot \wt{p}_2^{\sharp}(z^m z^n)\\
                             &=\bigoplus_{\substack{n\in N \\ \langle (e_i,0),(m+\?{p}_1(n),n)\rangle \geq 0}} \kk\cdot z^{(m+\?{p}_1(n),n)},
\end{align*}
where $\langle\cdot,\cdot\rangle$ in the last line is the dual pairing between $N_{\prin}$ and $M_{\prin}$.  Since $\langle (e_i,0),(m+\?{p}_1(n),n)\rangle = \val_{D_{(e_i,0)}}(z^{(m+\?{p}_1(n),n)})$, we see that the last line above is equal to $\s{O}_{\s{A}_{\prin}^{\star}}(m)(\wt{U}_i)$, as desired.

For the claim over $\phi$, recall from Remark \ref{Ii} that $E_{i,\phi}$ and $D_{i,\phi}$ are cut out by the same equations as $E_i$ and $D_i$, just reduced modulo $\s{I}_{\phi}$.  The claim now follows from the commutativity of \eqref{AprinCommutes}.
\end{proof}

We now want to take a quotient of $\s{R}_{X,\Lambda}$ which identifies the degree $m$ and degree $m'$ parts whenever $m-m'\in \Lambda_0$ (i.e., whenever $\s{O}_X(m)\cong \s{O}_X(m')$).  Let $\rho:\Lambda_0\rar K(X)^{\times}$ be a homomorphism of Abelain groups such that for each $m_0\in\Lambda_0$ and $m\in \Lambda$, we have $\rho(m_0)\s{O}_X(m) \subseteq \s{O}_X(m+m_0)$.  This gives what one calls a ``shifting family'' ($\rho_{\lambda}$ in \cite[Construction 4.3]{GHK3} corresponds to multiplication by our $\rho(\lambda)z^{\lambda}$).  One then defines a sheaf $\s{I}_{X,\Lambda,\rho}\subset \s{R}_{X,\Lambda}$ on $X$ by taking $\s{I}_{X,\Lambda,\rho}(U)$ to be the ideal generated by elements of the form $f-f\rho(m_0)z^{m_0}$ for $f\in \s{R}_{X,\Lambda}(U)$ and $m\in \Lambda_0$.  Finally, we can define the {\bf universal torsor} to be
\begin{align*}
\UT_X:=\SPEC \s{R}_{X,\Lambda}/\s{I}_{X,\Lambda,\rho}.
\end{align*}

\begin{thm}\label{UTor}
For any $\phi \in T_{K_1^*}$ and any $t\in i^{-1}(\phi)\subset T_M$, $\wt{p}_2^{\sharp}$ induces an isomorphism $\UT_{\s{X}_{\phi}^{\star}}\cong \s{A}_{t}^{\star}$.
\end{thm}
We note that in the special case of a skew seed without frozen variables and with $N_I=N$, this and Theorem \ref{R} together are essentially \cite[Thm. 4.4]{GHK3}.
\begin{proof}
Given Theorem \ref{R}, this amounts to showing that generators for the ideal sheaf $\s{I}_t$ cutting out $\s{A}_t^{\star}$, as given in \eqref{It}, are $\wt{p}_2^{\sharp}$ of generators for a sheaf $\s{I}_{\s{X}_{\phi}^{\star},M_I,\rho}$ on $\s{X}_{\phi}^{\star}$ arising from a shifting family $\rho$.

Recall the sheaf $\s{I}_{\phi}$ from \eqref{Iphi}.  Choose any section $s$ of $\?{p}_1$, and consider the shifting family corresponding to $\rho(m_0):=z^{s(m_0)}(t)z^{-s(m_0)}$, $m_0\in \?{p}_1(N)\subset M_I$, where $z^{s(m_0)}(t)$ denotes the evaluation of $z^{s(m_0)}$ at $t$.  The corresponding sheaf $\s{I}_{\s{X}^{\star}_{\phi},M_I,\rho}\subset \s{R}_{\s{X}^{\star}_{\phi},M_I}$ is, in terms of the coordinate system for $U_0$, globally generated by elements of the form 
\begin{align}\label{IXM}
    z^{m-m_0}z^{n+s(m_0)}-z^{m-m_0}z^nz^{s(m_0)}(t)z^{m_0} \mbox{~~(mod $\s{I}_{\phi}$)}
\end{align} 
for $m\in M_I$, $n\in N$, and $m_0\in \?{p}_1(N)$.  The relations induced by $\s{I}_{\phi}$ ensure that this does not depend on the choice of $s$.  Applying $\wt{p}_2^{\sharp}$ to \eqref{IXM} yields
\begin{align*}
    z^{(m+\?{p}_1(n),n+s(m_0))}-z^{(m+\?{p}_1(n),n)}z^{s(m_0)}(t).
\end{align*}
This is a generator for $\s{I}_t$ as given in \eqref{It}, and all the generators from \eqref{It} can be obtained in this way by varying the choices of $m$, $n$, $m_0$, and $s$.  Hence, $\wt{p}_2^{\sharp}(\s{I}_{\s{X}^{\star}_{\phi},M_I,\rho}) = \s{I}_t$, as desired.
\end{proof}

 For $X$, $\Lambda$, and $\rho$ as above, we define the {\bf Cox ring} of $X$ to be
 \begin{align*}
     \Cox(X):=\Gamma(X,\UT_X)
 \end{align*}
 as in \cite[Def. 4.5]{GHK3}.  If $\Pic(X)$ is torsion free, this is equivalent to $\bigoplus_{\nu \in \Pic(X)} \Gamma(X,\s{L}_{\nu})$.  We have the following immediate corollary of Theorem \ref{UTor}.
 
\begin{cor}\label{CoxCor}
For any $\phi \in T_{K_1^*}$ and any $t\in i^{-1}(\phi)\subset T_M$, $\wt{p}_2|_{\s{A}_t^{\star}}$ induces an isomorphism $$\wt{p}_2|_{\s{A}_t^{\star}}^{\sharp}:\Cox(\s{X}_{\phi}^{\star})\risom \Gamma(\s{A}_t^{\star},\s{O}_{\s{A}_{t}^{\star}}).$$
\end{cor}
 
 \begin{rmk}
Suppose $S$ is skew.  Then one can prove versions of Theorems \ref{R} and \ref{UTor} and Corollary \ref{CoxCor} for the full and finite-tree versions of the $\s{A}$- and $\s{X}$-spaces with various conditions on $S$ or the generality of $t$, exactly analogous to \cite[Thm. 4.4 and Cor. 4.6]{GHK3} but now allowing partial compactifications associated to frozen variables and not requiring $N_I=N$.
\end{rmk}
 
\subsection{Extension to the further partial compactifications}

Now consider a fan $\Sigma$ as in \S \ref{compact}.  If $\Sigma$ is non-singular (meaning that every cone is generated by part of a basis for $M$), then the above results all extend easily to the partial compactifications associated to $\Sigma$.  However, when $\Sigma$ is not smooth, $\s{X}^{\Sigma}$ will not be smooth either, and as a result, only a proper subset of Weil divisors will necessarily be Cartier.  Hence, while $W$ and $W_{\phi}$ from \eqref{Wm} will still give isomorphisms of the divisor class groups of $\s{X}^{\Sigma}$ and $\s{X}_{\phi}^{\Sigma}$ with $\coker(\?{p}_1)$, $\Pic(\s{X}^{\Sigma})$ and $\Pic(\s{X}^{\Sigma}_{\phi})$ will only correspond to a proper sublattice.  Namely, if $M_{\Sigma}$ denotes the sublattice of $M_I$ which is identified with Cartier divisors by $W$ (equivalently, by $W_{\phi}$), then $W$ and $W_{\phi}$ identify $M_{\Sigma}/(\?{p}_1(N)\cap M_{\Sigma})$ with $\Pic(\s{X}^{\Sigma})$ and $\Pic(\s{X}_{\phi}^{\Sigma})$, respectively.

Now recall that $\s{O}_{\s{A}_{\prin}^{\Sigma}}$ is $M_I$-graded.  Let $\s{O}_{\?{\s{A}}_{\prin}^{\Sigma}}$ denote the subsheaf spanned by homogeneous elements with degree in $M_{\Sigma}$.  Letting $\wt{p}_2^{\Sigma}$ denote the extension of $\wt{p}_2$ to a map $\s{A}_{\prin}^{\Sigma}\rar \s{X}_{\Sigma}$, we have
\begin{align}\label{RtoASigma}
    (\wt{p}_2^{\Sigma})^{\sharp}:\s{R}_{\s{X}^{\Sigma},M_{\Sigma}} \risom (\wt{p}_2^{\Sigma})_* \s{O}_{\?{\s{A}}_{\prin}^{\Sigma}}.
\end{align}

Let $\?{\s{A}}_{\prin}^{\Sigma}:= \SPEC(\s{O}_{\?{\s{A}}_{t}^{\Sigma}})$.  It follows from \eqref{RtoASigma} that $\wt{p}_2^{\Sigma}$ factors through $\?{\s{A}}_{\prin}^{\Sigma}$.  We write $p_2^{\Sigma}$ for the factor $\?{\s{A}}_{\prin}^{\Sigma} \rar \s{X}_{\Sigma}$.

Using the same shifting family as in Theorem \ref{UTor}, we get 
\begin{align*}
    \?{\s{A}}_{t}^{\Sigma} \cong \UT_{\s{X}_{\phi}^{\Sigma}}
\end{align*}
for $t\in i^{-1}(\phi)$, and so
\begin{align*}
   \Cox(\s{X}_{\phi}^{\Sigma})\cong \Gamma(\?{\s{A}}_{t}^{\Sigma},\s{O}_{\?{\s{A}}_{t}^{\Sigma}}).
\end{align*}

In summary:
\begin{thm}\label{SigmaThm}
There is a sublattice $M_{\Sigma}\subset M_I$ isomorphic to $\Pic(\s{X}^{\star})$ and $\Pic(\s{X}^{\Sigma}_{\phi})$.  If $\Sigma$ is non-singular, $M_{\Sigma}=M_I$.  Taking global Spec of the subsheaf of $\s{O}_{\s{A}^{\Sigma}_{\prin}}$ whose sections are $M_{\Sigma}$-graded gives a scheme $p_2^{\Sigma}:\?{\s{A}}^{\Sigma}_{\prin}\rar \s{X}^{\Sigma}$ with an isomorphism $(p_2^{\Sigma})^{\sharp}:\s{R}_{\s{X}^{\Sigma},M_{\Sigma}} \risom (p_2^{\Sigma})_* \s{O}_{\?{\s{A}}_{\prin}^{\Sigma}}$.  Furthermore, for any $\phi \in T_{K_1^*}$ and any $t\in i^{-1}(\phi)\subset T_M$, $\?{\s{A}}_{t}^{\Sigma}\cong\UT_{\s{X}_{\phi}^{\Sigma}}$, hence $\Cox(\s{X}_{\phi}^{\Sigma})\cong \Gamma(\?{\s{A}}_{t}^{\Sigma},\s{O}_{\?{\s{A}}_{t}^{\Sigma}})$.
\end{thm}

\section{Theta bases for line bundles}\label{ThetaLineBundles}

For a skew seed $S$, the main construction of \cite{GHKK}, applied to a scattering diagram in $M_{\prin}$ with initial scattering functions $1+z^{p_{1,\prin}(e_i)}$, $i\in I_{\uf}$, gives a linearly independent collection $\{\vartheta_q\}_{q\in M_{\prin}}$ of functions in some formal completion of $\kk[M_{\prin}]$ (i.e., some of the functions $\vartheta_q$ might be {\it formal} Laurent series rather than just Laurent polynomials).  These can be viewed as functions on some formal limit of $\s{A}_{\prin}^{\circ}$, or on $\s{A}_{\prin}^{\circ}$ itself when they are actually Laurent polynomials (recall from \S \ref{mutaskew} that by $\s{A}_{\prin}^{\circ}$ we mean what we call $\s{A}_{\prin}^{\star}$ but without any of the boundary divisors that we associated to the frozen vectors).

The definition of the theta functions and the fact that the initial scattering functions are homogeneous of $\deg=0$ (since $\deg(z^{p_{1,\prin}(e_i)})=0$) implies that $$\deg(\vartheta_q)=\deg(z^q)$$ (a fact also used in \cite[Construction 7.11]{GHKK}).  Furthermore, since $z^{(0,n)}$ is invariant under the initial scattering automorphisms (cf. \eqref{muAprin}) for each $n\in N$, the same is true of all wall-crossings for the scattering diagram used for constructing the theta functions.  It follows that $\vartheta_{(0,n)}=z^{(0,n)}$ for each $n\in N$, and moreover, 
\begin{align}\label{mn}
    \vartheta_{(m,n)}=z^{(0,n)}\vartheta_{(m,0)} 
\end{align}
for each $(m,n)\in M_{\prin}$.  Hence, for each $i\in F$ and $(m,n)\in M_{\prin}$, 
\begin{align}\label{valn}
    \val_{D_{e_i}}(\vartheta_{(m,n)}) = \val_{D_{e_i}}(\vartheta_{(m,0)}).
\end{align}

Let 
\begin{align}\label{Xi}
    \Xi:=\{q\in M_{\prin}|\vartheta_q\in \Gamma(\s{A}_{\prin}^{\star},\s{O}_{\s{A}_{\prin}^{\star}})\}.
\end{align}
Equivalently, $\Xi$ consists of the $q\in M_{\prin}$ for which $\vartheta_q$ can be expressed as a Laurent polynomial (rather than just a formal Laurent series) and $\val_{D_{e_i}}(\vartheta_q)\geq 0$ for each $i\in F$.  It follows from \eqref{mn}, \eqref{valn}, and the fact that $\vartheta_0:=1$ that $(0,N)\subset \Xi$ and that $\Xi$ is closed under addition by $(0,N)$.

\cite{Man3New} describes a generalization of the \cite{GHKK} construction which gives theta functions associated to seeds which are not necessarily skew.  The above definitions and claims still make sense and hold by the same arguments. 

\begin{ass}\label{FullFG}
$\{\vartheta_q\}_{q\in \Xi}$ spans $\Gamma(\s{A}_{\prin}^{\star},\s{O}_{\s{A}_{\prin}^{\star}})$ over $\kk$.
\end{ass}
It is unknown whether or not Assumption \ref{FullFG} always holds, cf. \cite[Question 7.17]{GHKK}.   For skew seeds, \cite{GHKK} says that ``the full Fock-Goncharov conjecture holds'' for $\s{A}_{\prin}^{\circ}$ when $\{\vartheta_q\}_{q\in M_{\prin}}$ forms an additive ($\kk$-vector space) basis for $\Gamma(\s{A}_{\prin}^{\circ},\s{O}_{\s{A}_{\prin}^{\circ}})$.  See \cite[Prop. 0.14]{GHKK} for a number of conditions which imply the full Fock-Goncharov conjecture, including $S$ being acyclic or admitting a maximal Green sequence.  If the full Fock-Goncharov conjecture holds for $\s{A}_{\prin}^{\circ}$, and if one makes the additional assumption that every frozen index in $S$ has an optimized seed (cf. \cite[Def. 9.1]{GHKK}), then \cite[Lemma 9.10(2)]{GHKK} implies that Assumption \ref{FullFG} holds.\footnote{Analogous results for non-skew seeds are not currently known.}  In particular, \cite{GHKK} claims that these conditions hold for the cluster structures on Grassmannians, on a maximal unipotent subgroup $N\subset \SL_r$, on the basic affine space $\SL_r/N$ (cf. \cite{Magee}), and on $(\SL_r/N)^3/\SL_r$.

When Assumption \ref{FullFG} does hold, we also get theta bases for each $\s{A}_{\phi}^{\star}$ as follows.  Let $s$ be any section of $\?{p}_1$.  For $\phi\in T_{K_1^*}$ and $(m,n_0)\in M_I\oplus \?{p}_1(N)$, define $\vartheta_{(m,n_0),\phi,s}:=\vartheta_{(m,s(n_0))}|_{\s{A}_{\phi}^{\star}}$.  Then $\Theta_{\Xi,\phi,s}:=\{\vartheta_{(m,n_0),\phi,s}\}_{(m,s(n_0))\in \?{p}_1(\Xi)}$ is an additive basis for $\s{A}_{\phi}^{\star}$. 

We note that $\vartheta_{(m,n_0),\phi,s}$ is independent of $s$ up to scaling.  More precisely, for a second section $s'$ of $\?{p}_1$, $\vartheta_{(m,n_0),\phi,s'}=z^{s'(n_0)-s(n_0)}(\phi)\vartheta_{q,\phi,s}$.  

Note that $\deg(\vartheta_{(m,n_0),\phi,s})=m-n_0$.  Recall that $\Xi$ is closed under addition by $(0,N)$.  Let $\?{\Xi}$ denote the projection of $\Xi$ to $M_I$ (equivalently, the intersection with $(M_I,0)$).  Given $\lambda \in M_I$, let $$\Xi_{\lambda}:=(\lambda+\?{p}_1(N))\cap \?{\Xi} \subset M_I$$ and let $$\vartheta_{m,\lambda,\phi,s}:=\vartheta_{(m,m-\lambda),\phi,s}=\left. z^{s(m-\lambda)}\vartheta_{(m,0)}\right|_{\s{A}_{\phi}^{\star}}.$$  Note that $\{\vartheta_{m,\lambda,\phi,s}\}_{m\in \Xi_{\lambda}}$ is exactly the set of degree $\lambda$ elements of $\Theta_{\Xi,\phi,s}$.

Theorem \ref{R} now easily implies the following:
\begin{thm}\label{ThetaBasis}
Suppose $S$ satisfies Assumption \ref{FullFG}.  Given $\lambda\in M_I$, let $\s{L}_\lambda:=\s{O}(W(\lambda))$ denote the corresponding line bundle on $\s{X}^{\star}$.  Under the identification $\wt{p}_2^{\sharp}$ of \eqref{Req}, $\{\vartheta_q\}_{q\in \Xi \cap \deg^{-1}(\lambda)}$ forms an additive basis for $H^0(\s{X}^{\star},\s{L}_\lambda)$.

Similarly, for $\phi \in T_{K_1^*}$, let $\s{L}_{\lambda,\phi}:=\s{L}_\lambda|_{\s{X}_{\phi}^{\star}}$.  Fix a section $s$ of $\?{p}_1$.  Then under the identification $p_{2,\phi}^{\sharp}$ of \eqref{Rphi}, we have that $\{\vartheta_{m,\lambda,\phi,s}\}_{m\in \Xi_{\lambda}}$ forms an additive basis for $H^0(\s{X}_{\phi}^{\star},\s{L}_{\lambda,\phi})$.
\end{thm}

In particular, by taking $\lambda=0$, we find that $\{\vartheta_{(m,n)}|(m,n)\in \Xi, m=\?{p}_1(n) \}$ is a $\kk$-module basis for $H^0(\s{X}^{\star},\s{O}_{\s{X}^{\star}})$.  Indeed, this agrees with the construction of theta functions on $\s{X}$ given in \cite[Construction 7.11]{GHKK}.

We note that Theorem \ref{SigmaThm} implies that Theorem \ref{ThetaBasis} analogously holds for $\s{X}^{\Sigma}$ and $\s{X}_{\phi}^{\Sigma}$ if one restricts to $\lambda \in M_{\Sigma}$.

\bibliographystyle{amsalpha}  
\bibliography{biblio}        
\index{Bibliography@\emph{Bibliography}}%

\end{document}